\documentclass[a4paper,11pt]{amsart}
\usepackage{amsmath,amssymb,amsthm}
\usepackage{amscd}
\usepackage{array,enumerate}
\newtheorem{Thm}{Theorem}[section]
\newtheorem{Prop}[Thm]{Proposition}
\newtheorem{Lem}[Thm]{Lemma}
\newtheorem{Cor}[Thm]{Corollary}
\newtheorem{Thmint}{Theorem}[section]
\newtheorem{MThm}{Main}

\newtheorem{Corint}[Thmint]{Corollary}
\theoremstyle{definition}
\newtheorem{Rem}[Thm]{Remark}

\newtheorem{Def}[Thm]{Definition}

\newcommand{\Cs}{\mbox{C$^\ast$}}

\DeclareMathOperator{\supp}{supp}
\DeclareMathOperator{\cl}{cl}
\DeclareMathOperator{\inte}{int}

\title[Crossed products of stable rank one]{Almost finiteness for general {\'e}tale groupoids and its applications to stable rank of crossed products}
\author{Yuhei Suzuki}
\subjclass[2000]{Primary~46L05, Secondary~54H20}
\keywords{Stable rank one, almost finiteness, reduced crossed products, pure infiniteness.}
\address{Graduate school of mathematics, Nagoya University, Chikusaku, Nagoya, 464-8602, Japan}
\email{yuhei.suzuki@math.nagoya-u.ac.jp}
\begin{document}
\begin{abstract}
We extend Matui's notion of almost finiteness to general {\'e}tale groupoids
and show that the reduced groupoid C$^\ast$-algebras of minimal almost finite groupoids have stable rank one.
The proof follows a new strategy,
which can be regarded as a local version of the large subalgebra argument.

The following three are the main consequences of our result.
(i) For any group of (local) subexponential growth and for any its minimal action admitting a totally disconnected free factor, the crossed product
 has stable rank one.
(ii) Any countable amenable group admits a minimal action
on the Cantor set all whose minimal extensions form the crossed product of stable rank one.
(iii) For any amenable group,
the crossed product of the universal minimal action has stable rank one.
\end{abstract}
\maketitle
\section{Introduction}\label{Sec:intro}
In the seminal paper \cite{Rie}, Rieffel introduced
the notion of (topological) stable rank for C$^\ast$-algebras.
The stable rank is introduced as a C$^\ast$-algebraic analogue of the Lebesgue covering dimension.
This in fact detects a topological complexity of C$^\ast$-algebras
in the sense that when a given C$^\ast$-algebra has stable rank one (the smallest possible value),
its K-theory behaves well.

In the celebrated paper \cite{Vil}, Villadsen showed
that simple stably finite C$^\ast$-algebras can have arbitrary stable rank.
In fact, Villadsen has constructed such C$^\ast$-algebras as AH-algebras.
His work suggests that even for relatively tractable C$^\ast$-algebras,
it would be hard to compute their stable rank.

In the present paper, we develop a new strategy to compute the stable rank of certain
groupoid C$^\ast$-algebras.
This strategy may be regarded as a local version of the large subalgebra argument (cf.~\cite{Put89}, \cite{Put00}, \cite{Phi05}, \cite{Phi14}, \cite{EN}).
This strategy is originally invented by Putnam in \cite{Put89}.
To describe the difference between the large subalgebra argument and
our strategy, let us briefly recall how the large subalgebra argument works.
For more detailed explanations and applications of large subalgebras,
we refer the reader to the introduction of the paper \cite{Phi14}.
The strategy can be divided into three steps.
Firstly, we find a large C$^\ast$-subalgebra.
Secondly, we study the structure of the large subalgebra.
Finally, we analyze the reduced crossed product based on the large subalgebra.
For the last step, Phillips \cite{Phi14} recently developed general theories.
For the first step, when the group is $\mathbb{Z}$, large subalgebras of the crossed products
can be found as Putnam's orbit-breaking subalgebras \cite{Put89}, \cite{Phi14}.
For the group $\mathbb{Z}^d$; $d\in \mathbb{N}$,
Phillips \cite{Phi05} has found that Forrest's result \cite{For} is sometimes useful to find a large
subalgebra of the crossed product.
However, Forrest's result \cite{For} heavily depends on the fact that
$\mathbb{Z}^d$ is a lattice of the Euclidean space $\mathbb{R}^d$.
Therefore, for groups beyond finitely generated abelian groups (even for $\mathbb{Z}^\infty$),
the first step of the strategy would be very hard to carry out.
Our strategy allows us to avoid this difficulty---
instead of finding a single large subalgebra,
we find a family of homogeneous C$^\ast$-algebras tending to large in a sense.

Inspired by the work of Latr{\'e}moli{\`e}re--Ormes \cite{LO}, Matui \cite{Mat12} has introduced the notion of almost finiteness
for a totally disconnected {\'e}tale groupoid.
He then studies how the homology groups and the topological full group
of an almost finite groupoid reflect dynamical properties of the groupoid.
For other applications of almost finiteness in this direction,
we refer the reader to \cite{Mat15} and \cite{Mat16}.
We extend to general {\'e}tale groupoids the notion of almost finiteness
and adapt this notion in our strategy.
Consequently, we obtain the following Main Theorem of the present paper.

\begin{MThm}\label{Thm:main}
Let $G$ be a minimal almost finite {\'e}tale groupoid.
Then the reduced groupoid C$^\ast$-algebra ${\rm C}^\ast_{\mathrm r}(G)$ has stable rank one.
\end{MThm}
In Appendix \ref{Sec:ES}, we show that any countable non-torsion amenable group
admits continuously many minimal actions on a prescribed space whose transformation groupoids are almost finite and pairwise non-isomorphic.
See Theorem \ref{Thm:examples} for the precise statement.
This is already enough to conclude that the class of groupoids satisfying the assumptions of the Main Theorem is fairly large.
We also remark that the minimality assumption of the Main Theorem is not removable.
Indeed, any compact space is almost finite as a groupoid,
while its groupoid \Cs-algebra does not have stable rank one
unless its covering dimension is less than $2$ \cite{Rie}.

We close this section by presenting applications of the Main Theorem.
Recall that a finitely generated group $\Gamma$ is said to have subexponential growth
if some (equivalently, any) symmetric finite generating set $S$ of $\Gamma$ fails to have a positive constant $C>1$
satisfying $\sharp(S^n)\geq C^n$ for all $n \in \mathbb{N}$.
We say that a group $\Gamma$ is of local subexponential growth
if it can be realized as a directed union of groups of subexponential growth.
It is not hard to see that this property is stable under
taking direct sums, inductive limits, subgroups, quotients.
Typical examples of groups of subexponential growth are virtually nilpotent groups, but it is also known to exist non-virtually-nilpotent groups of subexponential growth.
The first such a group is constructed by Grigorchuk \cite{Gri}.
Thanks to a recent remarkable result of Downarowicz--Zhang (Corollary 6.4 of \cite{DZ}), we conclude the following consequence of the Main Theorem to topological dynamical systems of groups of local subexponential growth.
\begin{Corint}\label{Corint:abe}
Let $\Gamma$ be a group of local subexponential growth.
Let $\alpha$ be a minimal action of $\Gamma$
on a compact space $X$ which factors onto a free action on a totally disconnected space.
Then the reduced crossed product $C(X)\rtimes _{{\mathrm r}, \alpha} \Gamma$ has stable rank one.
\end{Corint}
In the integer group case, this statement was announced in \cite{AP}.
Phillips also announced in the same paper that this result also holds true for the finitely generated free abelian groups.
Corollary \ref{Corint:abe} includes these statements,
and we believe that even for these specific cases,
our new proof is conceptually and practically simpler.
(Note that for abelian groups, we only need Lemma 6.3 of \cite{Mat12},
whose proof is much simpler than Forrest's theorem \cite{For}.)

We also remark that, even in the simplest case $\Gamma=\mathbb{Z}$,
Giol--Kerr \cite{GK} has constructed a minimal action
which is pathological in the following sense.
Its crossed product fails to have the following regularity properties; real rank zero, strict comparison, finite nuclear dimension, finite decomposition rank, tracial AF, Jiang--Su algebra stable. (For the definitions and backgrounds of these properties, we refer the reader to \cite{Win}.)
Since Giol--Kerr's action is obtained as an extension of the universal odometer transformation, it satisfies the assumptions of Corollary \ref{Corint:abe}.
This means that
\begin{enumerate}
\item
the stable rank one condition obtained in the Main Theorem
is almost the best possible result under our mild assumption,
\item
our work provides a general framework to analyze groupoid \Cs -algebras beyond the classifiable class in the sense of Elliott's program (see \cite{Win} for the precise meaning).
\end{enumerate}
As the classification theorem of {\em classifiable} \Cs-algebras is now almost complete \cite{Win} (modulo the UCT condition, which is automatic for amenable groupoid \Cs-algebras \cite{HK}, \cite{Tu99}),
our results would motivate next interesting targets of structure theory of \Cs-algebras.

We also have the following result for general countable amenable groups.
\begin{Corint}\label{Corint:ame}
Any countable infinite amenable group $\Gamma$ admits a minimal free action on the Cantor set
whose all minimal extensions have the crossed product of stable rank one.
\end{Corint}
To the best knowledge of the author, Corollary \ref{Corint:ame} is a new phenomenon in
the study of topological dynamical systems of amenable groups.
Minimal actions appearing in the statement can be generic minimal actions
as shown in \cite{Ker}.
We recall that any amenable minimal action admits
minimal extensions on various spaces \cite{GW}, \cite{Suz15}. See Proposition 2.1 in \cite{Suz15} for the precise statement.
In Appendix \ref{Sec:pi}, we give a non-amenable analogue of Corollary \ref{Corint:ame}.

The Main Theorem also reveals the structure of the crossed product of the universal minimal action for amenable groups.
Recall that any group $\Gamma$ admits a minimal action on a compact space
which factors onto every minimal action on a compact space.
In \cite{Ell}, Ellis showed that this universal property characterizes a unique minimal action up to conjugacy.
This action is called the universal minimal action of $\Gamma$.
The universal minimal action is known to be a free (minimal) action on a Stonean space.
\begin{Corint}\label{Corint:universal}
For any amenable group,
the crossed product of the universal minimal action has stable rank one.
\end{Corint}
In Appendix \ref{Sec:pi} we give a non-amenable analogue of Corollary \ref{Corint:universal} (see Corollary \ref{Cor:univpi}).
\subsection*{Organization of the paper}
Section \ref{Sec:pre} is the preliminary section.
We also fix the notations in this section.
In Section \ref{Sec:af}, we introduce and study the notion of almost finiteness for
general {\'e}tale groupoids.
In Section \ref{Sec:Main}, we prove the Main Theorem.
In Section \ref{Sec:app}, we discuss applications of the Main Theorem to the crossed products
of topological dynamical systems.
We further provide two appendices.
In Appendix \ref{Sec:pi}, based on a technique developed in Section \ref{Sec:Main},
we study pure infiniteness of the reduced groupoid C$^\ast$-algebras.
In particular, we solve the question of the pure infiniteness
of the reduced crossed product for minimal extensions.
In Appendix \ref{Sec:ES}, for any space in a certain large class,
we show that any countable non-torsion amenable group admits
a family of continuously many minimal free actions on it whose transformation groupoids are almost finite and pairwise non-isomorphic.
To distinguish groupoids, we introduce a new invariant which we call the eigenvalue set.
We also discuss a non-amenable analogue of this result.

\section{Preliminaries}\label{Sec:pre}
In this section, we recall some terminologies used throughout the paper.
We first fix the notations used in the paper.

\subsection*{Notations}
\begin{itemize}
\item For a set $S$, $\sharp S$ denotes the cardinality of $S$.
\item For a subset $S$ of a set $X$,
$\chi_S$ denotes the characteristic function of $S$.
\item For a subset $S$ of a topological space $X$,
the symbols $\cl(S)$ and $\inte(S)$ stand for
the closure and interior of $S$ respectively.
\item
For a map $\tau \colon A \rightarrow B$ between
subsets of a set $X$ and for a subset $U\subset X$,
we define $\tau(U):=\tau(U\cap A)$.
Similarly, for such a map $\tau$ and a function $f\colon X\rightarrow \mathbb{C}$,
we define the function $f\circ \tau \colon X\rightarrow \mathbb{C}$ to be
\[
(f\circ \tau)(s):=\left\{ \begin{array}{ll}
(f\circ \tau)(s) &{\rm if\ }s\in A\\
0 & {\rm otherwise}.\\
\end{array} \right.
\]
\item For an action $\alpha$ of a group $\Gamma$ on a set $X$,
we denote $\alpha_s(x)$ by $sx$ for short.
Similarly, for subsets $S\subset \Gamma$ and $U\subset X$,
we denote $\bigcup_{s\in S} \alpha_s(U)$ by $SU$.
\item For a probability space $(X, \mu)$ and a $\mu$-integrable function $f$,
we denote the integral $\int_{X} f d\mu$ by $\mu(f)$ for short.
\item For a compact space $X$ and an open subset $U \subset X$,
we identify $C_0(U)$ with the ideal of $C(X)$ consisting of
all functions vanishing on $X\setminus U$ in the canonical way.
\item For a C$^\ast$-algebra $A$ and $n\in \mathbb{N}$,
$M_n(A)$ denotes the C$^\ast$-algebra of all $n$-by-$n$ matrices over $A$.

\end{itemize}
\subsection{Groupoids}\label{SubSec:Groupoids}
Here we briefly recall basic facts on groupoids.
We refer the reader to Section 5.6 of \cite{BO} or \cite{Ren} for proper treatments.

A groupoid $G$ is a set equipped with
the subset $G^{(0)}$ of $G$ (the unit space), the maps $r, s\colon G \rightarrow G^{(0)}$
(the range and source maps),
and the map 
\[G^{(2)}:=\left\{(g, h) \in G^2 : s(g) = r(h)\right\} \rightarrow G;
(g, h) \mapsto gh~
{\rm (the~ groupoid~ multiplication)}\]
satisfying the following axioms.
\begin{enumerate}
\item For any $g\in G^{(0)}$, $r(g)=s(g)=g$.
\item For any $(g, h) \in G^{(2)}$, we have
$r(gh)=r(g)$ and $s(gh)=s(h)$.
\item For any $g\in G$, $r(g) g=gs(g) =g$.
\item For any pair $(g_1, g_2), (g_2, g_3)$ in $G^{(2)}$, $(g_1 g_2)g_3= g_1(g_2 g_3)$.
\item For any $g \in G$, there is a (unique) element $g^{-1} \in G$
with $g g^{-1}=r(g), g^{-1} g=s(g)$.
\end{enumerate}
A topological groupoid $G$ is a groupoid equipped with a topology
which makes all operation maps continuous.
An {\`e}tale groupoid is a topological groupoid whose range and source maps are local homeomorphisms. By this condition, the unit space is clopen in an {\`e}tale groupoid.

Throughout the paper, groupoids are assumed to be locally compact, Hausdorff,
{\'e}tale, and their unit spaces are assumed to be compact (possibly non-metrizable). 

Motivating examples of groupoids in the present work are coming from topological dynamical systems.
Let $\alpha$ be an action of a group $\Gamma$ on a compact Hausdorff space $X$.
The transformation groupoid $X\rtimes _\alpha \Gamma$ of $\alpha$ is the groupoid obtained
by equipping the product space $\Gamma \times X$ with the following groupoid structure.
The unit space is $\{e\} \times X$ which we often identify with $X$.
The range and source maps $r$ and $s$ are defined to be
$r(g, x):=\alpha_g(x)$ and $s(g, x):=x$ respectively.
Then, for a composable pair $((g, x), (h, y))$ (i.e., in the case $x=\alpha_h(y)$,)
the composite $(g, x)(h, y)$ is defined to be $(gh, y)$.

Let $G$ be a groupoid.
Let $C_{\mathrm c}(G)$ denote the space of compactly supported complex valued continuous functions on $G$.
We equip $C_{\mathrm c}(G)$ with the convolution product.
Since we do not consider the pointwise product,
we simply denote the convolution product of two functions $f, g\in C_{\mathrm c}(G)$
by $fg$. Namely, for $f, g \in C_{\mathrm c}(G)$ and $x\in G$,
\[(fg)(x)=\sum_{yz=x} f(y)g(z).\]
The involution on $C_{\mathrm c}(G)$ is given by
$f^\ast (x)=\overline{f(x^{-1})}$ for $f\in C_{\mathrm c}(G)$ and $x\in G$.
These operations make $C_{\mathrm c}(G)$ a $\ast$-algebra with the unit $\chi_{G^{(0)}}$.
The reduced groupoid C$^\ast$-algebra C$^\ast _{\mathrm r}(G)$ of $G$
is the unique C$^\ast$-completion of the $\ast$-algebra $C_{\mathrm c}(G)$
which admits a (unique) faithful conditional expectation
\[E\colon {\rm C}^\ast _{\mathrm r}(G) \rightarrow C(G^{(0)})\]
satisfying $E(f)=f|_{G^{(0)}}$ for all $f\in C_{\mathrm c}(G)$.
We refer to $E$ as the canonical conditional expectation of ${\rm C}^\ast _{\mathrm r}(G)$.
We regard $C_{\mathrm c}(G)$ as a $\ast$-subalgebra of ${\rm C}^\ast_{\mathrm r}(G)$.
With this identification, positivity and norm make a sense
for elements of $C_{\mathrm c}(G)$.
We note that for the transformation groupoid of a given action,
its reduced groupoid \Cs -algebra
is naturally identified with the associated reduced crossed product \Cs-algebra.

A subset $U$ of $G^{(0)}$ is said to be $G$-invariant
if it satisfies $r(GU)= U$.
Here for two subsets $V, W$ of a groupoid $G$,
define
\[VW:=\{vw: v\in V, w\in W, s(v)=r(w)\}.\]
When $W$ is a singleton $\{ w\}$,
we simply denote $VW$ by $Vw$. 
A groupoid $G$ is said to be minimal if
there is no proper open $G$-invariant subset of $G^{(0)}$.
A groupoid $G$ is said to be principal (resp.~essentially principal) if
the set (resp.~the interior of the set)
\[\{g\in G\setminus G^{(0)}: r(g)=s(g)\}\]
is empty.
These notions have a strong connection with the simplicity of the reduced groupoid C$^\ast$-algebras as follows.
When the reduced groupoid C$^\ast$-algebra is simple,
then obviously the groupoid must be minimal.
Conversely, when a groupoid is minimal and essentially principal,
the reduced groupoid C$^\ast$-algebra is known to be simple \cite{AS}.

A subset $V$ of $G$ is said to be a $G$-set if
both the range and source maps are injective on $V$.
For any compact $G$-set $V$,
the formula $\tau_V:=r\circ (s|_V)^{-1}$
defines a homeomorphism
from $s(V)$ onto $r(V)$.
Set
\[ [[G]]:=\{\tau_V \colon V{\rm\ is\ a\ compact\ open\ }G\mathchar`-{\rm set\ with\ }r(V)=s(V)=G^{(0)}\}.\]
Then $[[G]]$ is a group of homeomorphisms on $G^{(0)}$.
We call $[[G]]$ the topological full group of $G$.
Note that for totally disconnected groupoids,
this definition coincides with the standard definition (cf.~\cite{GPS99}, \cite{Mat15} for instance).

A probability regular Borel measure $\mu$ on $G^{(0)}$
is said to be $G$-invariant if
for any compact $G$-set $V$ and any Borel subset $B$ of $s(V)$,
we have $\mu(\tau_V(B))=\mu(B)$.
Note that by the uniqueness statement of the Riesz representation theorem,
one can show that a probability regular Borel measure $\mu$ on $G^{(0)}$
is $G$-invariant if and only if
for any compact $G$-set $V$ and any continuous function $f\in C(G^{(0)})$
supported on $r(V)$,
we have $\mu(f)=\mu(f\circ \tau_V)$.
Let $M(G)$ denote the set of all $G$-invariant probability measures.
Obviously $M(G)$ is compact with respect to the weak-$\ast$ topology.
Every $\mu \in M(G)$ defines a tracial state
on the reduced groupoid C$^\ast$-algebra by the formula $\mu\circ E$.
Since $E$ is faithful, the resulting tracial state is faithful if and only if
the support of $\mu$ is $G^{(0)}$.

For a nonempty closed subset $U$ of $G^{(0)}$, set
\[G_U :=r^{-1}(U)\cap s^{-1}(U).\]
Then the relative topology makes $G_U$ an {\'e}tale groupoid.

A subset $U$ of $G^{(0)}$ is said to be $G$-full
if it satisfies $r(GU)=G^{(0)}$.
Note that for any open $G$-full subset $U$,
by the compactness of $G^{(0)}$,
one can find a finite sequence $V_1, \ldots, V_n$ of compact $G$-sets
satisfying $\bigcup_{i=1}^n \tau_{V_i}(U)=G^{(0)}$.
Indeed, since $G$ is \`etale and locally compact,
one can choose a family $(V_\lambda)_{\lambda \in \Lambda}$ of compact $G$-sets
satisfying $\bigcup_{\lambda \in \Lambda} \inte(V_\lambda)= G$.
This implies 
$\bigcup_{\lambda \in \Lambda} \inte(\tau_{V_\lambda}(U))=G^{(0)}$.
Thus the family $(V_\lambda)_{\lambda \in \Lambda}$ contains the desired sequence.

\subsection{Factors and extensions}
Here we briefly recall basic terminologies of topological dynamical systems.
From now on, the symbol $\Gamma$ stands for a discrete group (possibly uncountable), and the symbols $X, Y$ stand for compact Hausdorff spaces (possibly nonmetrizable).
Actions of $\Gamma$ on $X, Y$ are always assumed to be continuous.

An action $\alpha$ of $\Gamma$ on $X$ is said to be minimal
if the transformation groupoid is minimal.
An $\alpha$ is said to be free (resp.~essentially free)
if the transformation groupoid is principal (resp.~essentially principal).
For two actions $\alpha$, $\beta$ of a group $\Gamma$ on compact spaces $X$, $Y$,
$\alpha$ is said to be an extension of $\beta$ (or $\alpha$ extends $\beta$)
if there is a $\Gamma$-equivariant quotient map $\pi \colon X\rightarrow Y$.
In this case $\beta$ is said to be a factor of $\alpha$ (or $\beta$ factors onto $\alpha$). We say that an extension is a minimal extension if it is a minimal action.

\subsection{Stable rank}
Here we recall the definition and basic facts of stable rank for unital C$^\ast$-algebras.
We refer the reader to \cite{Rie} for further information.
\begin{Def}
Let $A$ be a unital C$^\ast$-algebra.
For each $n\in \mathbb{N}$,
define
\[{\rm Lg}_n(A):=\{(a_1, \ldots, a_n) \in A^n : A a_1 + \cdots + A a_n=A\}.\]
Then the stable rank of $A$ is defined to be the value
\[\min \{n\in \mathbb{N}: {\rm Lg}_n(A){\rm \ is \ norm\ dense\ in\ } A^n\}.\]
\end{Def}
Obviously, the stable rank takes the value in $\{ 1, 2, \ldots, \infty\}$.
For any $n\in \{1, 2, \ldots, \infty\}$, Villadsen \cite{Vil} has constructed
a simple unital separable AH-algebra of stable rank $n$.

For the stable rank one, Rieffel has found the following useful characterization.
\begin{Prop}[\cite{Rie}, Proposition 3.1]
A unital C$^\ast$-algebra $A$ has stable rank one if and only
if the set of invertible elements of $A$ is norm dense in $A$.
\end{Prop}
Rieffel gave the following striking application of the stable rank one condition in K-theory.
We refer the reader to \cite{Bla} for basic facts on K-theory.
\begin{Thm}[\cite{Rie}]\label{Thm:srK}
Let $A$ be a unital C$^\ast$-algebra of stable rank one.
Let $n\in \mathbb{N}$ be given.
Then the following statements hold true.
\begin{enumerate}[\upshape (1)]
\item For two projections $p, q$ of $M_n(A)$, they are unitary equivalent in $M_n(A)$
if and only if we have $[p]_0=[q]_0$ in K$_0(A)$.
\item For two unitary elements $u, v$ of $M_n(A)$,
they sit in the same connected component of the unitary group of $M_n(A)$
if and only if
we have $[u]=[v]$ in K$_1(A)$.
\end{enumerate}
\end{Thm}

This emphasizes the importance to determine when a given C$^\ast$-algebra
has stable rank one.
Note that even for commutative \Cs-algebras, these properties fail in general.
This is one of the major difficulty of the classification problem of vector bundles
up to isomorphism (even on manifolds),
in spite of the fact that $K$-groups are computable invariants.
See also Remark \ref{Rem:rr0} for applications to the real rank of certain groupoid \Cs-algebras.
\subsection{Connected components and separation by clopen subsets}\label{ss:conn}
Let $X$ be a compact Hausdorff space.
For $x\in X$,
define $\langle x \rangle$
to be the intersection of all clopen neighborhoods of $x$.
It is not hard to show that the set $\langle x \rangle$ coincides with the connected component of $X$ containing $x$.
The following lemmas are immediate consequences of the definition.
\begin{Lem}\label{Lem:conn2}
Let $x\in X$ be given.
Let $C$ be a closed subset of $X$ disjoint from $\langle x \rangle$.
Then there is a clopen subset $U \subset X$
satisfying $\langle x\rangle \subset U \subset X \setminus C$.
\end{Lem}
\begin{Lem}\label{Lem:conn1}
Let $x_1, \ldots, x_n $ be elements of $X$
satisfying $\langle x_i \rangle \neq \langle x_j \rangle$ for $i\neq j$.
Then there are pairwise disjoint clopen subsets $U_1, \ldots, U_n$ of $X$
satisfying $x_i \in U_i$ for each $i$.
\end{Lem}
\section{Almost finite groupoids}\label{Sec:af}
In this section, we introduce a notion of almost finiteness for general (\'etale) groupoids. This notion is originally introduced for totally disconnected groupoids by Matui \cite{Mat12}.
We then establish basic properties of almost finite groupoids which are used in Section \ref{Sec:Main}.

We first introduce a few terminologies needed to define almost finiteness.
\begin{Def}
Let $G$ be a groupoid.
Let $C$ be a compact subset of $G$ and let $\epsilon>0$.
Let $K$ be a compact subgroupoid of $G$ containing the unit space $G^{(0)}$.
We say that $K$ is $(C, \epsilon)$-invariant
if the following inequality holds for all $s\in G^{(0)}$.
\[\frac{\sharp(CKs \setminus Ks)}{\sharp(Ks)}< \epsilon.\]
\end{Def}
\begin{Def}\label{Def:fd}
Let $K$ be a compact groupoid.
We say that a clopen subset $F$ of $K^{(0)}$ is a fundamental domain of $K$
if there is a partition $K^{(0)}=\bigsqcup_{i=1}^n \bigsqcup_{j=1}^{N_i} F^{(i)}_{j}$
of $K^{(0)}$ by clopen subsets
satisfying the following conditions.
\begin{enumerate}[\upshape (1)]
\item $F=\bigsqcup_{i=1}^n F^{(i)}_1$.
\item For each $i\in\{1, \ldots, n\}$ and $j, k\in \{1, \ldots, N_i\}$,
there is a unique compact open $K$-set $V^{(i)}_{j, k}$
satisfying $r(V^{(i)}_{j, k})=F^{(i)}_{j}$ and $s(V^{(i)}_{j, k})=F^{(i)}_{k}$.
\item $K=\bigsqcup_{i=1}^n\bigsqcup_{1 \leq j, k \leq N_i} V^{(i)}_{j, k}$.
\end{enumerate}
\end{Def}
\begin{Rem}\label{Rem:fd}
The groupoid $K$ appearing in Definition \ref{Def:fd}
is isomorphic to the groupoid $\coprod_{i=1}^n F^{(i)}_1\times (\mathbb{Z}_{N_i}\rtimes \mathbb{Z}_{N_i})$.
Here for each $i$, the group $\mathbb{Z}_{N_i}$ acts on itselt by the left translations.
From this description, it is easy to show that the analogue of Lemma 6.1 in \cite{Mat12}
is valid for compact groupoids admitting a fundamental domain.
A groupoid of this form is said to be elementary in \cite{Ren} (III. Definition 1.1).
From this particular form, it is not hard to conclude that a clopen subset $U$ of $K^{(0)}$
contains a fundamental domain if and only if it satisfies $\mu(U)>0$ for all $\mu \in M(K)$.
\end{Rem}
\begin{Def}\label{Def:ele}
Let $G$ be a groupoid.
We say that a subgroupoid $K$ of $G$ is elementary
if $K$ contains the unit space $G^{(0)}$ and
admits a fundamental domain.
\end{Def}
\begin{Rem}
In \cite{GPS04}, the authors introduced the notion of compact \`etale equivalence relation.
It follows from Lemma 3.4 of \cite{GPS04} that
our notion of an elementary subgroupoid
coincides with the notion of a compact \`etale equivalence subrelation in \cite{GPS04}.
\end{Rem}

We now introduce the definition of almost finiteness for general \'etale groupoids.
\begin{Def}[cf.~Definition 6.2 of \cite{Mat12}]
We say that a groupoid $G$ is almost finite if it satisfies the following conditions.
\begin{enumerate}[\upshape (1)]
\item
The union of all compact open $G$-sets covers $G$.
\item For any compact subset $C \subset G$ and $\epsilon>0$,
there is a $(C, \epsilon)$-invariant elementary subgroupoid $K$ of $G$.
\end{enumerate}
\end{Def}

We remark that, by Lemma 4.7 of \cite{Mat12}, for totally disconnected groupoids,
our definition of almost finiteness coincides with Matui's original one.

\begin{Rem}
For any group action on a compact space,
it is easy to see that the almost finiteness of the transformation groupoid implies the amenability of the acting group (cf.~ Lemma \ref{Lem:cas}).
For general groupoids, it is not known if almost finiteness implies amenability.
Also, it can be shown that almost finiteness implies the existence of
an invariant probability measure. (See Lemma 6.5 in \cite{Mat12} and Lemma \ref{Lem:prob2}.)
To the author's knowledge, this is the only known
obstruction of almost finiteness for amenable (minimal) principal groupoids.
\end{Rem}
We first study a connection between invariant probability measures of almost finite groupoids
and that of almost invariant elementary subgroupoids.

By modifying the proof of Lemma 6.5 in \cite{Mat12} in a straightforward way,
one can show the following lemma.
Here we omit the proof. See \cite{Mat12} for the details.

\begin{Lem}[cf.~Lemma 6.5 in \cite{Mat12}]\label{Lem:prob1}
Let $G$ be a groupoid.
Let $V$ be a compact $G$-set and let $\epsilon>0$.
Let $K$ be a $(V\cup V^{-1}, \epsilon)$-invariant elementary subgroupoid of $G$.
Then for any $\mu \in M(K)$, we have
\[|\mu(r(V\setminus K))|< \epsilon {\rm\ and\ } |\mu(s(V\setminus K))|<\epsilon.\]
In particular, for any Borel subset $A$ of $s(V)$, we have
\[|\mu(A)-\mu(\tau_V(A))|<2\epsilon.\]
\end{Lem}
\begin{Lem}[cf.~Lemma 6.5 of \cite{Mat12}]\label{Lem:prob2}
Let $G$ be an almost finite groupoid.
Let $(C_\lambda)_{\lambda\in \Lambda}$ be an increasing net
of compact subsets of $G$ satisfying
$\bigcup_{\lambda \in \Lambda} \inte(C_\lambda)=G$.
Let $(\epsilon_\lambda)_{\lambda\in \Lambda}$ be
a net of positive numbers decreasing to $0$.
For each $\lambda \in \Lambda$, let
a $(C_\lambda, \epsilon_\lambda)$-invariant elementary subgroupoid $K_\lambda$ of $G$
and $\mu_\lambda \in M(K_\lambda)$ be given.
Then any weak-$\ast$ cluster point of the net $(\mu_\lambda)_{\lambda \in \Lambda}$
is $G$-invariant.
In particular $M(G)$ is nonempty.
\end{Lem}
\begin{proof}
Let $\mu$ be a weak-$\ast$ cluster point of the net $(\mu_\lambda)_{\lambda \in \Lambda}$.
It suffices to show the equality $\mu(f)=\mu(f\circ \tau_V)$
for any $f\in C(G^{(0)})$ and
any compact $G$-set $V$ satisfying $\supp(f)\subset r(V)$.

Given such $f$ and $V$.
Let $\epsilon>0$ be given.
Take a partition $\supp(f)=\bigsqcup_{i=1}^n A_i$
of $\supp(f)$ by Borel sets and a sequence $z_1, \ldots, z_n \in \mathbb{C}$
satisfying $|f(s)-z_i|<\epsilon$ for all $s\in A_i$ and $i$.
For each $i$, put $B_i:=\tau_V^{-1}(A_i)$.
Then, by Lemma \ref{Lem:prob1}, one can find $\lambda \in \Lambda$
satisfying the following conditions.
\begin{enumerate}[\upshape (1)]
\item $|\mu(f)-\mu_\lambda(f)|< \epsilon.$
\item $|\mu(f\circ\tau_V)- \mu_\lambda(f\circ \tau_V)|<\epsilon.$
\item $\sum_{i=1}^n |z_i||\mu_\lambda(A_i)- \mu_\lambda(B_i)|<\epsilon$.
\end{enumerate}
By the choices of the partition $(A_i)_{i=1}^n$ and the sequence $(z_i)_{i=1}^n$, we have
\[\left|\mu_\lambda(f)-\sum_{i=1}^n z_i\mu_\lambda(A_i)\right|<\epsilon.\]
Since $\supp(f\circ \tau_V)=\tau_V^{-1}(\supp(f))$,
we have
\[\supp(f\circ \tau_V)=\bigsqcup_{i=1}^n B_i.\]
Moreover, for any $i\in \{1, \ldots, n\}$ and any $s\in B_i$,
we have $|(f\circ \tau_V)(s)-z_i|<\epsilon$.
This yields
\[\left|\mu_\lambda(f\circ \tau_V)-\sum_{i=1}^n z_i\mu_\lambda(B_i)\right|<\epsilon.\]
By combining these inequalities, we obtain
\begin{eqnarray*}
|\mu(f)-\mu(f\circ\tau_V)|&\leq& |\mu(f)-\mu_\lambda(f)|+|\mu_\lambda(f)-\mu_\lambda(f\circ\tau_V)|+|\mu(f\circ \tau_V)-\mu_\lambda(f\circ \tau_V)|\\
&<&4\epsilon +\sum_{i=1}^n |z_i||\mu_\lambda(A_i)-\mu_\lambda(B_i)|\\
&<&5\epsilon.\end{eqnarray*}
Since $\epsilon>0$ is arbitrary, we have
$\mu(f)=\mu(f\circ\tau_V)$ as desired.
\end{proof}
\begin{Lem}[cf.~Remark 6.6 in \cite{Mat12}]\label{Lem:ep}
Any minimal almost finite groupoid is essentially principal.
\end{Lem}
\begin{proof}
To lead to a contradiction, suppose we have a minimal almost finite groupoid $G$
which is not essentially principal.
Take a compact $G$-set $V\subset G\setminus G^{(0)}$
satisfying $\inte(V)\neq \emptyset$ and $r(g)=s(g)$ for all $g\in V$.
Take $\mu \in M(G)$.
For any $\epsilon>0$, take a $(V\cup V^{-1}, \epsilon)$-invariant
elementary subgroupoid $K$ of $G$.
Then, since $K$ is principal,
we have $K\cap V=\emptyset$.
This equality and Lemma \ref{Lem:prob1} yield $\mu(s(V))<\epsilon$.
Since $\epsilon>0$ is arbitrary, we have
$\mu(s(V))=0$. This contradicts the minimality of $G$.
\end{proof}
\begin{Lem}\label{Lem:Gsets}
Let $G$ be an almost finite groupoid.
Let $U$ be an open $G$-full set.
Then there is an elementary subgroupoid $K$ of $G$
which makes $U$ a $K$-full set.
\end{Lem}
\begin{proof}
To lead to a contradiction, assume that such a $K$ does not exist
for a $G$-full open set $U$.
We observe that for any elementary subgroupoid $K$ of $G$,
$U$ is $K$-full if and only if it satisfies $\mu(U)>0$ for all $\mu \in M(K)$.
Take nets $(C_\lambda)_{\lambda\in \Lambda}$, $(\epsilon_\lambda)_{\lambda\in \Lambda}$, and
$(K_\lambda)_{\lambda\in \Lambda}$ as in the statement of Lemma \ref{Lem:prob2}.
Then by assumption, for each $\lambda \in \Lambda$,
one can find $\mu_\lambda \in M(K_\lambda)$ satisfying $\mu_\lambda(U)=0$.
Let $\mu$ be a weak-$\ast$ cluster point of the net $(\mu_\lambda)_{\lambda \in \Lambda}$.
By Lemma \ref{Lem:prob2}, we have $\mu \in M(G)$.
Since each $\mu_\lambda$ is supported on $G^{(0)}\setminus U$,
so $\mu$ is. This contradicts the $G$-fullness of $U$.
\end{proof}
\begin{Lem}\label{Lem:full}
Let $G$ be an almost finite groupoid.
Let $U$ be a $G$-full clopen subset of $G^{(0)}$.
Then the restriction groupoid $G_U$ is almost finite.
\end{Lem}
\begin{proof}
Since $U$ is $G$-full,
one can take compact $G$-sets $V_1, \ldots, V_n$ satisfying
$r(V_i)\subset U$ for all $i$ and $\bigcup_{i=1}^n s(V_i)=G^{(0)}$.
Let a compact subset $C \subset G_U$ and $\epsilon >0$ be given.
We assume that $\epsilon \leq 1$.
Set $\widetilde{C}:=C \cup V_1 \cup V_2 \cup \cdots \cup V_n.$
We claim that for any $(\widetilde{C}, \epsilon/2n)$-invariant elementary
subgroupoid $K$ of $G$,
the restriction groupoid $K_U$ is a $(C, \epsilon)$-invariant elementary subgroupoid
of $G_U$. This completes the proof.

Let $K$ be a $(\widetilde{C}, \epsilon/2n)$-invariant elementary subgroupoid of $G$.
Let $s\in U$ be given.
To prove the claim, we first show the inequality
\[ \sharp(K_U s)\geq \frac{1}{2n}\sharp(Ks).\]
The equality $\bigcup_{i=1}^n s(V_i)=G^{(0)}$ yields
\[\sum_{i=1}^n \sharp(V_iKs)\geq \sharp (Ks).\] 
Hence one can choose $i \in \{1, \ldots, n\}$ satisfying
\[\sharp(V_iKs)\geq \frac{1}{n}\sharp(Ks).\]
By the choice of $K$ and the assumption $\epsilon \leq 1$,
we have \[\sharp(V_iKs\setminus Ks)\leq \frac{1}{2n}\sharp(Ks).\]
Since $V_iKs\subset G_U$, these two inequalities yield
\begin{eqnarray*}\sharp(K_U s)&\geq& \sharp(V_iKs \cap Ks)\\
&\geq& \sharp(V_iKs)- \frac{1}{2n}\sharp(Ks)\\
&\geq& \frac{1}{2n}\sharp(Ks).
\end{eqnarray*}

Since
$CK_U s\setminus K_U s \subset \widetilde{C}Ks\setminus Ks$,
we conclude
\begin{eqnarray*}
\frac{\sharp(CK_U s\setminus K_U s)}{\sharp (K_U s)}&\leq&2n\frac{\sharp(\widetilde{C}K s\setminus K s)}{\sharp (Ks)}\\
&<&\epsilon.
\end{eqnarray*}
\end{proof}
The following statement is useful to find minimal almost finite groupoids.
\begin{Lem}\label{Lem:invsub}
Let $G$ be an almost finite groupoid.
Let $C$ be a closed $G$-invariant subset of $G^{(0)}$.
Then the restriction groupoid $G_C$ is almost finite.
\end{Lem}
\begin{proof}
Let $D$ be a compact subset of $G_C$ and let $\epsilon>0$.
Take a $(D, \epsilon)$-invariant elementary subgroupoid $K$
of $G$. Since $C$ is $G$-invariant and closed,
the restriction groupoid $K_C$ provides a $(D, \epsilon)$-invariant elementary subgroupoid of $G_C$.
\end{proof}

Now it is convenient to introduce the following notion of equivalence for
clopen subsets of the unit space of a groupoid $G$. 
We say that two clopen subsets $U, V$ of $G^{(0)}$ are equivalent in $G$
if there is a compact open $G$-set $W$
satisfying $r(W)=U$ and $s(W)=V$. We denote this equivalence relation by $\sim$.

As a consequence of Lemma \ref{Lem:full}, we have the following divisibility property of
a clopen subset of a minimal almost finite groupoid.
This property plays an important role in the proof of the Main Theorem.

\begin{Lem}\label{Lem:div}
Let $G$ be a minimal almost finite groupoid of infinite cardinality.
Let $U\subset G^{(0)}$ be a clopen subset.
Then, for any $N\in \mathbb{N}$, there is a partition
$U=\bigsqcup_{i=1}^n\bigsqcup_{j=1}^{N_i}U^{(i)}_{j}$ of $U$ by clopen subsets
such that $U^{(i)}_{j} \sim U^{(i)}_{k}$ for all $i, j, k$, and that $N_i\geq N$ for all $i$.
\end{Lem}
\begin{proof}
By Lemma \ref{Lem:full}, we only need to show the statement in the case $U=G^{(0)}$.
In this case, the statement is equivalent to
the existence of an elementary subgroupoid $K$ of $G$
satisfying $\sharp(Ks)\geq N$ for all $s\in G^{(0)}$.

To find such a subgroupoid, take a compact subset $C$ of $G$
satisfying $\sharp(Cs)\geq 2N$ for all $s\in G^{(0)}$.
(To find such a $C$, take an increasing net $(C_\lambda)_{\lambda \in \Lambda}$ of compact open subsets
of $G$ whose union covers $G$.
Then, for each $\lambda\in \Lambda$,
the set
\[U_\lambda:=\{ s\in G^{(0)}: \sharp(C_\lambda s)\geq 2N\}\]
is clopen, and the union of the increasing net $(U_\lambda)_{\lambda \in \Lambda}$ covers $G^{(0)}$.
By the compactness of $G^{(0)}$,
for sufficiently large $\lambda \in \Lambda$, the compact set $C_\lambda$ satisfies the desired condition.)
Take a $(C, 1)$-invariant elementary subgroupoid $K$ of $G$.
Then, for any $s\in G^{(0)}$, we have
\[\sharp(Cs)-\sharp(Ks)\leq \sharp(CKs\setminus Ks)\leq \sharp(Ks).\]
This yields $\sharp (Ks)\geq N$ for all $s\in G^{(0)}$.
\end{proof}

\begin{Lem}[cf.~\cite{Mat12}, Lemma 6.7; see also Lemma 2.5 of \cite{GW} and Lemma 3.5 of \cite{GPS99}]\label{Lem:com}
Let $G$ be an almost finite groupoid.
Let $U$ and $V$ be clopen subsets of $G^{(0)}$ satisfying
$\mu(U)<\mu(V)$ for all $\mu \in M(G)$.
Then there is an element $\varphi \in [[G]]$
with $\varphi(U) \subset V$.
\end{Lem}
\begin{proof}
As mentioned in Definition \ref{Def:fd}, Lemma 6.1 of \cite{Mat12} is valid
for any groupoid admitting a fundamental domain.
This fact and Lemma \ref{Lem:prob2} allow us to prove
the claim in the same way as Lemma 6.7 of \cite{Mat12}.
See \cite{Mat12} for the details.
\end{proof}
\section{Stable rank of \Cs -algebras of almost finite groupoids}\label{Sec:Main}
In this section, we prove the Main Theorem.
Since the Main Theorem is obviously valid for groupoids of finite cardinality,
throughout this section, we suppose that groupoids are of infinite cardinality.

We first prove the following key lemma, which is about zero divisors of the reduced groupoid C$^\ast$-algebras.
Recall that an element $a$ of a ring $R$ is
called a right (resp.~left) zero divisor
if there is a nonzero element $x\in R$ satisfying $xa=0$ (resp.~$ax=0$).
An element of $R$ is called a two-sided zero divisor if
it is both a right and a left zero divisor.
\begin{Lem}\label{Lem:zerodiv}
Let $G$ be a minimal almost finite groupoid.
Let $a\in {\rm C}^\ast _{\mathrm r}(G)$ be a right zero divisor.
Then for any $\epsilon >0$, there are a nonempty clopen subset $U \subset G^{(0)}$
and a unitary element $u\in {\rm C}^\ast _{\mathrm r}(G)$ satisfying
$\|\chi_U ua \| <\epsilon$.
The analogous statement also holds true for left zero divisors.
\end{Lem}
\begin{proof}
Obviously it suffices to show the claim for right zero divisors.
Let a right zero divisor $a\in {\rm C}^\ast_{\mathrm r}(G)$ and $\epsilon>0$ be given.
We may assume $\|a\|=1$ and $\epsilon<1$.
Choose an element $x\in {\rm C}^\ast _{\mathrm r}(G)$ of norm one with $xa=0$.
By replacing $x$ by $x^\ast x$, we may assume $x$ is positive.
Since the canonical conditional expectation $E$ is faithful, we have $E(x)\neq 0$.

Let $\eta$ be a positive number less than $\rho:=\frac{1}{2}\|E(x)\|$ (determined explicitly later).
Choose a positive element $x_0\in C_{\mathrm c}(G)$
satisfying $\|x - x_0\| <\eta$ and $\|x_0\|= 1$.
Note that $\|E(x_0)\|>\rho$.
Choose compact open $G$-sets
$V_1, \ldots, V_n$ whose union covers $\supp(x_0)$.
Take $s\in G^{(0)}$ satisfying $E(x_0)(s)>\rho$.
Then, by Lemma \ref{Lem:div}, one can choose a clopen neighborhood $U\subset G^{(0)}$ of
$s$ satisfying
\[W:=\bigcup_{i=1}^n \tau_{V_i}(U)\neq G^{(0)}.\]
Note that $W$ is a clopen subset of $G^{(0)}$.
From the definition of $W$,
we have $\chi_W x_0\chi_U x_0=x_0\chi_U x_0$.
This implies the inequality $x_0\chi_U x_0 \leq \chi_{W}$.
Also, it follows from a direct computation that
$E(x_0 \chi_U x_0)(s)> \rho^2$.
Define
\[y:=\|x_0\chi_U x_0\|^{-1} x_0\chi_U x_0 \in C_{\mathrm c}(G).\]
We then have $\|E(y)\| > \rho^2$ and $\|y a\|<\rho^{-2}\eta$.
Therefore, by choosing $\eta= \rho^4 \epsilon$,
we obtain a positive element $y\in C_{\mathrm c}(G)$ of norm one
with the following conditions.
\begin{enumerate}[\upshape (1)]
\item $\|E(y) \|> \rho^2$.
\item $\|ya\|<\rho^2 \epsilon$.
\item There is a proper clopen subset $W \subset G^{(0)}$ satisfying $y\leq \chi_W$.
\end{enumerate}
Note that $E(y^2)\geq E(y)^2$ hence $\|E(y^2)\|>\rho^4$.

Next we will find a nonnegative function $f\in C(G^{(0)})$ satisfying the following conditions.
\begin{enumerate}[\upshape (i)]
\item $\|f\|\leq \rho^{-2}$.
\item $fy^2 f\in C(G^{(0)})$.
\item $\|fy^2f\|=1$.
\item The set $U:=\inte(\{ s\in G^{(0)}:(fy^2f)(s)=1\})$ is nonempty.
\end{enumerate}
To construct such a function, we observe first that since $G$ is essentially principal, one can choose $s\in G^{(0)}$
satisfying $y^2(s)>\rho^4$ and
\[r^{-1}(\{s\})\cap s^{-1}(\{s\})\cap \supp(y^2)=\{s\}.\]
Since $\supp(y^2)$ is compact, one can find
an open neighborhood $V \subset G^{(0)}$ of $s$
satisfying
\[r^{-1}(V)\cap s^{-1}(V)\cap \supp(y^2)\subset G^{(0)}.\]
This relation implies
 \[C_0(V) \cdot y^2 \cdot C_0(V) \subset C(G^{(0)}).\]
Take a nonnegative function $g\in C_0(V)$ satisfying
\[(y^2(s))^{-1/2} < g(s) \leq \|g\| \leq \rho^{-2}.\]
We then have $gy^2 g\in C(G^{(0)})$
and $(gy^2g)(s)>1$.
Set $D:=\{t\in G^{(0)}:(gy^2g)(t)\geq 1\}$.
Note that $D$ is a closed subset of $G^{(0)}$
containing $s$ in its interior.
Define \[h:=((gy^2g)|_D)^{-1/2} \in C(D).\]
Take a continuous extension $\tilde{h} \in C(G^{(0)})$ of $h$
satisfying $0 \leq \tilde{h} \leq 1$.
Then the function $f:=g\tilde{h} $ satisfies the desired conditions.

Take a nonempty open subset $U_a$ of $U$ satisfying
$\cl(U_a)\subset U$.
We also fix an element $s\in G^{(0)} \setminus W$.
Put $S:=\langle s \rangle$ (see Section \ref{ss:conn} for the notation).
By Lemma \ref{Lem:Gsets}, one can find compact open $G$-sets $W_1, \ldots, W_m$
satisfying $S \subset \bigcup_{i=1}^m{\tau_{W_i}(U_a)}$, $S\subset r(W_i)$ for all $i$,
and
$\tau_{W_i}^{-1}(S) \cap \tau_{W_j}^{-1}(S) =\emptyset$ for $i \neq j$.
Then by Lemma \ref{Lem:conn1}, one can take pairwise disjoint clopen subsets $Z_1, \ldots, Z_m$ of $G^{(0)}$
satisfying $\tau_{W_i}^{-1}(S) \subset Z_i$ for each $i$.
For each $i$, set $U_i:= U_a\cap Z_i$.
Then the closures of $U_1, \ldots, U_m$ in $G^{(0)}$ are pairwise disjoint.
Moreover, for each $i$, as $\tau_{W_i}^{-1}(S) \subset Z_i$, we have
$S \cap \tau_{W_i}(U_a) = S \cap \tau_{W_i}(U_i)$.
This yields $S\subset \bigcup_{i=1}^m \tau_{W_i}(U_i)$.
Since the union $\bigcup_{i=1}^m \tau_{W_i}(U_i)$ is open,
thanks to Lemma \ref{Lem:conn2}, one can find a clopen neighborhood $Z \subset G^{(0)}\setminus W$ of $s$ contained in $\bigcup_{i=1}^m {\tau_{W_i}(U_i)}$.

Choose nonnegative functions $\rho_1, \ldots, \rho_m \in C(G^{(0)})$
such that $\rho_i\rho_j=0$ for $i\neq j$, that $\rho_i \equiv 1$ on $U_i$ for each $i$, and that ${\rm supp}(\rho_i) \subset \cl(U)$ for each $i$.
Set
\[v:= \sum_{i=1}^m \chi_{W_i}\rho_i \in C_{\mathrm c}(G).\]
We then have
\begin{eqnarray*}
v fy^2fv^\ast&=&\sum_{1\leq i, j\leq m} \chi_{{W_i}} \rho_i fy^2f \rho_j \chi_{W_j}^\ast \\
&=&\sum_{i=1}^m \chi_{W_i}\rho_i^2 \chi_{W_i}^\ast\\
&=&\sum_{i=1}^m \rho_i^2\circ\tau_{W_i}^{-1}\\
&\geq& \chi_Z.
\end{eqnarray*}
Moreover, this computation shows $vfy^2fv^\ast \in C(G^{(0)})$.
Hence there is a function $k\in C(G^{(0)})$ of norm at most one
satisfying $kvfy^2fv^\ast k^\ast =\chi_Z$.
Take such a function $k$ and put $w:=kv$.
We claim that $w$ is of norm one.
Indeed, since $\rho_i=\rho_i fy^2f$ for all $i$, we have $v=vfy^2f$.
This shows $w=wfy^2f$.
Hence we have
\[ww^\ast = w fy^2f w^\ast=\chi _Z.\]
In particular this yields $\|w\| = 1$.

Set $t:=yfw^\ast$.
Then $t$ is a partial isometry element of ${\rm C}^\ast _{\mathrm r}(G)$ satisfying
$t^\ast t=\chi_Z$ and $tt^\ast\leq \chi_W$.
Since $Z$ and $W$ are disjoint,
the projections $t^\ast t$ and $tt^\ast$ are orthogonal.
Hence the element
\[u:=t+t^\ast+(1-t^\ast t-tt^\ast)\]
is a unitary element of ${\rm C}^\ast _{\mathrm r}(G)$.
Moreover, by definition, we have \[u^\ast\chi_Z u =tt^\ast= yfw^\ast wfy.\]
This shows 
\begin{eqnarray*} \|\chi_Z ua\|&=&\|yfw^\ast wfy a\|\\
&\leq&\|yf\| \|w^\ast w\| \|f\| \|ya\| \\
&<&\epsilon.
\end{eqnarray*}
Hence $Z$ and $u$ satisfy the required condition.
\end{proof}
The next lemma can be seen as a very special case of Green's imprimitivity theorem \cite{Gre}.
Since we need an explicit isomorphism, we include a proof.
\begin{Lem}\label{Lem:imp}
Let $K$ be a compact groupoid with a fundamental domain.
Let $K^{(0)}=\bigsqcup_{i=1}^n \bigsqcup_{j=1}^{N_i}{F^{(i)}_{j}}$
be a partition of $K^{(0)}$
as in Definition \ref{Def:fd}.
Then there is a $\ast$-isomorphism
\[\Phi\colon {\rm C}^\ast _{\mathrm r} (K) \rightarrow \bigoplus_{i=1}^n M_{N_i}(C(F^{(i)}_1)).\] 
\end{Lem}
\begin{proof}
For each $i\in \{1, \ldots, n\}$ and $j, k\in \{1, \ldots, N_i\}$, let $V_{j, k}^{(i)}$ be the unique compact open $K$-set
satisfying $r(V_{j, k}^{(i)})=F^{(i)}_{j}$ and $s(V_{j, k}^{(i)})=F^{(i)}_{k}$.
We then have
\[K=\bigsqcup_{i=1}^n \bigsqcup_{1 \leq j, k \leq N_i} V^{(i)}_{j, k}.\]
This shows that the set
\[\left\{ f\chi_{V_{j, k}^{(i)}}: 1\leq i\leq n, 1\leq j, k\leq N_i, f\in C(F^{(i)}_{j})\right\}\]
linearly spans $C(K)={\rm C}^\ast _{\mathrm r}(K)$.
For each $i \in \{1, \ldots, n\}$, let $(e_{j, k}^{(i)})_{1 \leq j, k\leq {N_i}}$ denote
the standard matrix units of
$M_{N_i}(\mathbb{C}) \subset M_{N_i}(C(F^{(i)}_1))\subset \bigoplus_{i=1}^n M_{N_i}(C(F^{(i)}_1))$.
Here the first inclusion is the canonical unital one.
Define a linear map
\[\Phi\colon {\rm C}^\ast _{\mathrm r}(K) \rightarrow \bigoplus_{i=1}^n M_{N_i}(C(F^{(i)}_1))\]
to be
\[\Phi(f\chi_{V_{j, k}^{(i)}}):= (f\circ \tau_{V_{j, 1}^{(i)}})\cdot e_{j, k}^{(i)}\]
for $i, j, k$, and $f\in C(F^{(i)}_{j})$.
Then direct computations show
that $\Phi$ is a $\ast$-isomorphism.
\end{proof}
Before going to the proof, we introduce a notation.
Suppose $G$ is an essentially principal groupoid and let $\varphi \in [[G]]$ be given.
Take a compact open $G$-set $V$ satisfying $\varphi=\tau _V(=r \circ (s|_V)^{-1})$.
Since $G$ is essentially principal, the choice of $V$ is unique.
We then define $u_\varphi :=\chi_V \in C_{\rm c}(G)$.
A direct computation shows that $u_\varphi$ is a unitary element of ${\rm C}^\ast _{\mathrm r}(G)$.
\begin{proof}[Proof of the Main Theorem]
By Lemma \ref{Lem:prob2}, the reduced groupoid \Cs -algebra ${\rm C}^\ast _{\mathrm r}(G)$
admits a faithful tracial state.
Hence, thanks to Proposition 3.2 of \cite{Ror91}, it suffices to show that
any two-sided zero divisor of ${\rm C}^\ast_{\mathrm r}(G)$ is contained in $\overline{GL}({\rm C}^\ast _{\mathrm r}(G))$,
where we denote by $\overline{GL}(A)$ the norm closure of the group of invertible elements of a \Cs-algebra $A$.

Let $a$ be a two-sided zero divisor of ${\rm C}^\ast_{\mathrm r}(G)$.
We may assume $a$ is of norm one.
Let $\epsilon >0$ be given.
Then by Lemma \ref{Lem:zerodiv}, there are unitary elements $u, v$ of ${\rm C}^\ast_{\mathrm r}(G)$ and
nonempty clopen subsets $U, V \subset G^{(0)}$ satisfying 
$\|\chi_U uav\|<\epsilon$ and $\|uav\chi_V\|< \epsilon$.
By Lemmas \ref{Lem:div} and \ref{Lem:com}, one can find a nonempty clopen subset $U_0$ of $U$ and $\varphi \in [[G]]$ satisfying $\varphi(U_0)\subset V$. 
Therefore, by replacing $U$ by $U_0$ and replacing $v$ by $vu_\varphi ^\ast$, we may assume $U=V$.
Choose $b_0\in C_{\mathrm c}(G)$ satisfying
$\|b_0\|\leq 1$ and $\|uav-b_0\|<\epsilon$.
Set \[b:=b_0-\chi_U b_0 -b_0\chi_U +\chi_U b_0\chi_U.\]
Since $\chi_U$ is a projection,
we have $\chi_U b=b\chi_U=0$.

By Lemma \ref{Lem:div}, one can find a nonempty clopen subset $U_1$ of $U$
satisfying $\mu(U_1) < \frac{1}{2} \mu(U)$ for all $\mu \in M(G)$.
Define \[\rho:=\inf\{\mu(U_1) : \mu\in M(G)\}.\]
Note that $\rho$ is a positive number since $M(G)$ is nonempty
(by Lemma \ref{Lem:div}) and weak-$\ast$ compact.

Take a compact open neighborhood $B$ of $\supp(b)$.
We claim that there is an elementary subgroupoid $K$ of $G$
satisfying the inequalities 
\[\mu(s(B\setminus K)) <\mu(U_1) < \frac{1}{2}\mu(U)\]
for all $\mu \in M(K)$.
To lead to a contradiction, assume that such a $K$ does not exist.
Set \[\Lambda:=\{(C, \eta): C\subset G {\rm\ compact}, \eta>0\}.\]
We equip $\Lambda$ with the partial order $\preceq $
defined by $(C, \eta)\preceq (C', \eta')$ if they satisfy the relations $C\subset C'$ and $\eta \geq \eta'$.
This makes $\Lambda$ a directed set.
For each $\lambda=(C, \eta) \in \Lambda$,
take a $(C, \eta)$-invariant elementary subgroupoid $K_\lambda$ of $G$.
Since $B$ is realized as the union of finitely many compact $G$-sets,
by Lemma \ref{Lem:prob1}, for sufficiently large $\lambda \in \Lambda$,
we have
\[\mu(s(B\setminus K_\lambda)) \leq \frac{1}{2}\rho\]
for all $\mu \in M(K_\lambda)$.
Therefore, without loss of generality, we may assume that
each $K_{\lambda}$ satisfies the above inequality.
By assumption, for each $\lambda\in \Lambda$, we must have a $K_\lambda$-invariant probability measure $\mu_\lambda$ satisfying
either \[\mu_\lambda(U_1) \leq \mu_\lambda(s(B\setminus K_\lambda)) \leq \frac{1}{2} \rho\] or \[\mu_\lambda(U_1)\geq\frac{1}{2}\mu_\lambda(U).\]
By Lemma \ref{Lem:prob2}, any weak-$\ast$ cluster point of
the net $(\mu_\lambda)_{\lambda\in \Lambda}$ is $G$-invariant.
Since both $U_1$ and $U$ are clopen, the above condition yields
that either $\mu(U_1) \leq \frac{1}{2}\rho$ or $\mu(U_1)\geq \frac{1}{2}\mu(U)$ holds for any weak-$\ast$ cluster point $\mu$ of $(\mu_\lambda)_{\lambda \in \Lambda}$.
In either case, this is a contradiction.

Take an elementary subgroupoid $K$ of $G$ satisfying the inequalities
\[\mu(s(B\setminus K)) < \mu(U_1) < \frac{1}{2}\mu(U)\]
for all $\mu \in M(K)$. Note that the inequalities trivially implies 
\[\mu(s(B\setminus K)\setminus U_1) <\mu(U_1)< \frac{1}{2}.\]
Since $s(B\setminus K)$ is a clopen subset of $G^{(0)}$,
by Lemma \ref{Lem:com}, one can find a clopen subset $V_1$ of $G^{(0)}\setminus U_1$
which is equivalent to $U_1$ in $K$
and satisfies
\[s(B\setminus K)\subset U_1\sqcup V_1.\]
Then observe that for any $\mu \in M(K)$,
\[\mu(U\setminus(U_1\sqcup V_1)) \geq \mu(U)-2\mu(U_1) >0.\] 
This implies that the clopen set $U\setminus (U_1\sqcup V_1)$
contains a fundamental domain of $K$.

Take an element $\psi \in [[G]]$ satisfying $\psi(U_1) =V_1$, $\psi(V_1)=U_1$, and $\psi(s)=s$ for all $s\in G^{(0)}\setminus (U_1 \sqcup V_1)$.
(To find such a $\psi$, first take a compact open $G$-set $D$
satisfying $r(D)=U_1$, $s(D)=V_1$.
Define \[E:= D \sqcup D^{-1} \sqcup (G^{(0)}\setminus (U_1 \sqcup V_1)).\]
Then a direct computation shows that $E$ is a compact open $G$-set
with $r(E)=s(E)=G^{(0)}$. Clearly $\psi:=\tau_E \in [[G]]$ satisfies the desired condition.)

Set $p:=\chi_{V_1}, q:= \chi_{U_1}$, and $w:=u_\psi$.
Set $W:=G^{(0)}\setminus (U_1\sqcup V_1)$ and $e:=1-p-q$.
Note that $we=ew=e$.
By the relation
\[s(\supp(b)\setminus K) ( \subset s(B\setminus K))\subset U_1\sqcup V_1,\]
we have $ebe \in {\rm C}^\ast_{\mathrm r}(K_W)$.
We also have $pwb=wqb=0$ and $wbq=0$.
Therefore the matrix form of the element $wb$
with respect to the decomposition $1=p+e+q$ is of the form
\[wb=\begin{pmatrix}
0 & 0 & 0\\
\ast & d & 0 \\
\ast & \ast & 0 \\
\end{pmatrix},\]
where $d:=ewbe=ebe$.

We claim that $d \in \overline{GL}({\rm C}^\ast_{\mathrm r} (G_W))$.
This yields that
$wb$, hence $b$, is contained in $\overline{GL}({\rm C}^\ast_{\mathrm r} (G))$.
Since \[\|b-uav\|\leq \|b-b_0\|+\|b_0-uav\|<7\epsilon,\]
this completes the proof.

To prove the claim, take a fundamental domain $Z$ of $K$ contained in $U\setminus (U_1 \sqcup V_1)$.
Since $Z$ is contained in $W$,
it is also a fundamental domain of $K_W$.
Note that since $Z \subset U$, we have
$\chi_Z b=b\chi_Z=0$.
Take a partition $W=\bigsqcup_{i=1}^n \bigsqcup_{k=1}^{N_i} Z^{(i)}_k$
of $W$ as in Definition \ref{Def:fd} for $K_W$ with $\bigsqcup_{i=1}^n Z^{(i)}_1=Z$.
Put $N:=\max\{N_i: i=1, \ldots, n\}$.
By Lemmas \ref{Lem:full} and \ref{Lem:div}, for each $i$, one can find a partition
\[Z^{(i)}_1=\bigsqcup_{j=1}^{M_i} \bigsqcup_{k=1}^{L_{i, j}}Z^{(i, j)}_k\]
of $Z^{(i)}_1$ by clopen subsets satisfying $Z^{(i, j)}_k\sim Z^{(i, j)}_l$ in $G_W$ and $L_{i,j}\geq N$ for all $i,j,k, l$.
For each $(i, j)$, take a compact open $G_W$-set $W_{i, j}$
satisfying the following conditions.
\begin{enumerate}
\item $r(W_{i, j}) = s(W_{i, j})= \bigsqcup_{k=1}^{L_{i, j}}Z^{(i, j)}_k$.
\item $\tau_{W_{i, j}}(Z^{(i, j)}_k)=Z^{(i, j)}_{k+1}$ for all $k$ (mod $L_{i, j}$).
\item The ${L_{i, j}}$-th power $(\tau_{W_{i, j}})^{L_{i, j}}$ is equal to the identity map on $\bigsqcup_{k=1}^{L_{i, j}}Z^{(i, j)}_k.$
\end{enumerate}
Define $L$ to be the (open) subgroupoid of $G_W$
generated by $K_W$ and $\bigcup_{i, j} W_{i, j}$.
Since $Z$ is a fundamental domain of $K_W$,
it is not hard to show that the $L$ is compact and
that the union $\bigsqcup_{i, j}Z^{(i, j)}_1$
is a fundamental domain of $L$.
Take a $\ast$-isomorphism
\[\Phi \colon {\rm C}^\ast _{\mathrm r}(L) \rightarrow \bigoplus_{i, j} M_{N_iL_{i, j}}(C(Z^{(i, j)}_1))\]
as in the proof of Lemma \ref{Lem:imp}.
The equalities $\chi_Z d =d \chi_Z=0$
show that $\Phi(d)$ is unitary equivalent to an element of the form
\[\bigoplus_{i, j}{\rm diag}(0_{L_{i, j}}, c_{i, j}),\]
where $c_{i, j}\in M_{L_{i, j}(N_i-1)}(C(Z^{(i, j)}_1))$ for each $(i, j)$.
Moreover, since $d\in {\rm C}^\ast _{\mathrm r}(K_W)$,
for each $i, j, k$, the $d$ commutes with the characteristic function of ${r(K_W Z^{(i, j)}_k)}$ (a central element in ${\rm C}^\ast _{\mathrm r}(K_W)$), hence with that of ${r(K_W Z^{(i, j)}_k)}\cap(W\setminus Z)$.
Since $Z$ is a fundamental domain of $K_W$, for each $i, j, k$, the intersection $r(K_W Z^{(i, j)}_k)\cap(W\setminus Z)$ decomposes into the disjoint union of $(N_i-1)$ clopen
subsets of $G^{(0)}$ each of which is equivalent to $Z^{(i, j)}_k$ in $K_W$ (hence in $L$).
This proves that $c_{i, j}$ is unitary equivalent to an element of the form
\[{\rm diag}(c_{i, j, 1}, \ldots, c_{i, j, L_{i, j}}),\] where $c_{i, j, k}\in M_{N_i -1}(C(Z^{(i, j)}_1))$ for all $k$.
Since $L_{i,j}\geq N>N_i -1$ for all $i, j$, Lemma 4.2 of \cite{EHT}
now proves
\[d \in \overline{GL}({\rm C}^\ast_{\mathrm r} (L))\subset \overline{GL}({\rm C}^\ast_{\mathrm r} (G_W)).\]
\end{proof}
\begin{Rem}\label{Rem:rr0}
The Main Theorem and Theorem \ref{Thm:srK} show that the reduced groupoid C$^\ast$-algebra of
an minimal almost finite groupoid has cancellation of projections.
From this, after minor modifications, one can adapt the arguments in Sections 3 and 4 of \cite{Phi05} to minimal almost finite totally disconnected groupoids.
Consequently, for such a groupoid $G$,
the reduced groupoid \Cs -algebra ${\rm C}^\ast _{\mathrm r}(G)$
has real rank zero and strict comparison.
\end{Rem}
\section{Applications of Main Theorem to topological dynamical systems}\label{Sec:app}
In this section, we give applications of the Main Theorem.
We first show the following permanence property
which is useful in applications of the Main Theorem.

\begin{Lem}\label{Lem:quo}
Let $G$ be a groupoid which admits a quotient homomorphism
$\pi$ from $G$ onto an almost finite groupoid $H$ with the following properties:
 $\pi$ is proper, and
the restriction map $\pi|_{Gs}\colon Gs\rightarrow H\pi(s)$ is bijective for each $s\in G^{(0)}$.
Then $G$ is almost finite.
\end{Lem}
Typical examples of groupoids satisfying the assumptions in Lemma \ref{Lem:quo} come from factor maps of topological dynamical systems. Consider compact spaces $X$ and $Y$ equipped with a $\Gamma$-action.
Then any $\Gamma$-equivariant quotient map $\pi \colon X\rightarrow Y$
extends to a quotient homomorphism
$\pi \colon X \rtimes _\alpha \Gamma \rightarrow Y \rtimes _\beta \Gamma$
by the formula $\pi(s, x):=(s, \pi(x))$.
This homomorphism obviously satisfies the two conditions in Lemma \ref{Lem:quo}.
\begin{proof}[Proof of Lemma \ref{Lem:quo}]
By the assumptions on $\pi$,
for any elementary subgroupoid $K$ of $H$,
the preimage $\pi^{-1}(K)$ is an elementary subgroupoid of $G$.
Furthermore, for any compact subset $C\subset H$ and $\epsilon>0$,
$K$ is $(C, \epsilon)$-invariant if and only if $\pi^{-1}(K)$ is $(\pi^{-1}(C), \epsilon)$-invariant.
\end{proof}
We now deal with topological dynamical systems of amenable groups.
We first recall some terminologies related to partitions of underlying spaces
of group actions (cf.~\cite{Ker}, beginning of Section 4).
Let $\alpha$ be an action of $\Gamma$ on $X$.
A tower of $\alpha$ is a pair $(S, B)$ of a nonempty finite subset $S\subset \Gamma$
and a nonempty subset $B\subset X$ such that the sets $sB$; $s\in S$, are pairwise disjoint.
The set $S$ is called the shape of the tower $(S, B)$.
A tower decomposition of $\alpha$ is a sequence $(S_1, B_1), \ldots, (S_n, B_n)$ of towers
satisfying $X= \bigsqcup_{i=1}^n S_iB_i$.
The sets $S_1, \ldots, S_n$ are called the shapes of the tower decomposition $((S_i, B_i))_{i=1}^n$.
Let $F$ be a finite subset of $\Gamma$ and let $\epsilon>0$.
A finite subset $S$ of $\Gamma$ is said to be $(F, \epsilon)$-invariant if
it satisfies
\[\sharp(FS\setminus S)<\epsilon \sharp S.\]
A tower decomposition $((S_i, B_i))_{i=1}^n$ is said to be $(F, \epsilon)$-invariant
if all the shapes $S_1, \ldots, S_n$ are $(F, \epsilon)$-invariant.
A tower decomposition $((S_i, B_i))_{i=1}^n$ is said to be clopen if each $B_i$ is clopen in $X$.

\begin{Lem}\label{Lem:cas}
Let $\alpha$ be an action of $\Gamma$ on $X$.
Let $F\subset \Gamma$ be a finite subset and let $\epsilon>0$.
Then the following conditions are equivalent.
\begin{enumerate}[\upshape(1)]
\item The groupoid $X\rtimes_{\alpha} \Gamma$ admits an $(F\times X, \epsilon)$-invariant elementary subgroupoid.
\item
The action $\alpha$ admits an $(F, \epsilon)$-invariant clopen tower decomposition.
\item
The action $\alpha$ admits an $(F, \epsilon)$-invariant dynamical tiling in the sense of Definition 6.1 in \cite{DZ}.
\end{enumerate}
Consequently, the following conditions are equivalent.
\begin{enumerate}[\upshape(i)]
\item The transformation groupoid $X\rtimes_\alpha \Gamma$
is almost finite.
\item The action $\alpha$ admits
clopen tower decompositions with arbitrary invariance.
\item The action $\alpha$ has the tiling property in the sense of Definition 6.1 in \cite{DZ}.
\end{enumerate}
\end{Lem}
\begin{proof}
We first show the equivalence of (1) and (2).
Put $G:= X\rtimes_\alpha \Gamma$.
To prove the claim, we first give a correspondence
between elementary subgroupoids of $G$ and clopen tower decompositions of $\alpha$.

Let $K$ be an elementary subgroupoid of $G$.
Take a partition
\[K=\bigsqcup_{i=1}^n \bigsqcup_{1\leq j, k \leq N_i} V^{(i)}_{j, k}\]
of $K$ as in Definition \ref{Def:fd}.
By discarding empty sets, we may assume that each $V^{(i)}_{j, k}$ is nonempty.
Also, by refining the partition if necessary,
we may assume that each $V^{(i)}_{j, k}$ is of the form $\{ s\} \times U$; $s\in \Gamma$, $U\subset X$.
Put $B_i:=V^{(i)}_{1, 1}$ for each $i$.
Let $\pi \colon G \rightarrow \Gamma$
denote the projection onto the first coordinate.
Define
\[S_i:=\pi(s^{-1}(B_i) \cap K) \subset \Gamma\] for each $i$.
Then it is clear from the definition
that the sequence $((S_i, B_i))_{i=1}^n$ is a clopen tower decomposition of $\alpha$.

Conversely, if we have a clopen tower decomposition $((S_i, B_i))_{i=1}^n$ of $\alpha$,
then the formula
\[K:=\bigsqcup_{i=1}^n \bigsqcup_{s, t\in S_i} \{st^{-1}\}\times tB_i\]
defines an elementary subgroupoid of $G$.

We observe that, up to refinements of clopen tower decompositions,
these correspondences are each other's inverse.
Now one can show that a clopen tower decomposition is $(F, \epsilon)$-invariant
if and only if the corresponding elementary subgroupoid is $(F, \epsilon)$-invariant.
In particular the conditions (1) and (2) are equivalent.

We next show the equivalence of (2) and (3).
To prove this, we give a correspondence
between clopen tower decompositions of $\alpha$ and dynamical tilings of $\alpha$.

First let $(S_1, B_1), \ldots, (S_n, B_n)$ be a clopen tower decomposition of $\alpha$.
By replacing $(S_i, B_i)$ by $(S_i s_i^{-1}, s_i B_i)$ for some $s_i \in S_i$ if necessary,
we may assume that each $S_i$ contains the unit $e$.
By replacing the terms $(S_i, B_i), (S_j, B_j)$ with $S_i =S_j$
by $(S_i, B_i \cup B_j)$ if necessary,
we may also assume that $S_1, \ldots, S_n$ are pairwise distinct.
Set $\mathcal{S}:=\{ S_1, \ldots, S_n\}$, $\Delta:= \mathcal{S} \cup \{ 0 \}$.
We define $\mathcal {T} \colon X \rightarrow \Delta ^\Gamma$ as follows.
For $x\in X$ and $s\in \Gamma$, set
\[\mathcal{T}_x(s):=\left\{ \begin{array}{ll}
S_i & {\rm~ when~} sx \in B_i,\\
0 & {\rm otherwise}.\\
\end{array} \right.
\]
(Note that $\Delta ^\Gamma$ is equipped with the right shift $\Gamma$-action.)
Since each $B_i$ is clopen, $\mathcal{T}$ is continuous.
As $(S_1, B_1), \ldots, (S_n, B_n)$ form a tower decomposition of $\alpha$,
each $\mathcal{T}_x$ associates a tiling of $\Gamma$ (see Definition 2.13 of \cite{DZ} for the definition).
The $\Gamma$-equivariance of $\mathcal{T}$ is clear from the definition.
We thus obtain a dynamical tiling $\mathcal{T}$ from a clopen tower decomposition.

Conversely, suppose we have a dynamical tiling $\mathcal{T}$ of $\alpha$
with shapes in $\mathcal{S}:= \{S_1, \ldots, S_n\}$.
For each $i$,
set \[B_i := \{ x\in X : \mathcal{T}_x(e)= S_i\}.\]
Then, by the continuity of $\mathcal{T}$, each $B_i$ is clopen.
Moreover, as each $\mathcal{T}_x$ is a tiling of $\Gamma$ with shapes in $\mathcal{S}$,
the sequence $(S_1, B_1), \ldots, (S_n, B_n)$
forms a tower decomposition of $\alpha$.

Obviously these correspondences preserve $(F, \epsilon)$-invariance.
We therefore conclude the equivalence of (2) and (3).
\end{proof}

Now one can prove Corollary \ref{Corint:abe}.
\begin{proof}[Proof of Corollary \ref{Corint:abe}]
We show that for any free (not necessary minimal) action
$\alpha$ of a local subexponential group $\Gamma$
on a totally disconnected space $X$, the transformation groupoid $X \rtimes _\alpha \Gamma$
is almost finite.
Once this claim is proved, the statement of Corollary \ref{Corint:abe} follows from the Main Theorem and Lemma \ref{Lem:quo}.

To show the claim, we first take an increasing net $(\Gamma_i)_{i\in I}$
of countable subgroups of $\Gamma$
whose union is $\Gamma$.
Obviously, when each $X \rtimes_\alpha \Gamma_i$ is almost finite,
then so $X \rtimes _\alpha \Gamma$ is.
Hence it suffices to show the claim when $\Gamma$ is countable.

In this case, by Lemma 3.1 of \cite{Suz17} (cf.~\cite{RS}),
one can find a factor $\beta$ of $\alpha$
which is free and whose underlying space is metrizable.
Now the almost finiteness of $X \rtimes_\alpha \Gamma$
follows from Lemmas \ref{Lem:quo}, \ref{Lem:cas}, and Corollary 6.4 in \cite{DZ}.
\end{proof}

Kerr recently announced the following theorem (now available in \cite{Ker}).
Here we recall that a property of an action is said to be generic
if the property is valid on a $G_\delta$-dense set of an appropriate Polish space
of actions. See \cite{Ker} for the precise statement.
For backgrounds and history of this subject,
see the introduction of the paper \cite{Hoc}.
Other connections between this subject and \Cs-algebra theory
can be found in \cite{Ker}, \cite{Suz15}, \cite{Suz17}, \cite{Suz18}.
\begin{Thm}[\cite{Ker}, Theorem 4.2]\label{Thm:Ker}
For any countable amenable group,
its generic minimal free actions on the Cantor set
admit clopen tower decompositions with arbitrary invariance.
\end{Thm}

Corollary \ref{Corint:ame} is an immediate consequence of the Main Theorem, Lemmas \ref{Lem:quo}, \ref{Lem:cas}, and Theorem \ref{Thm:Ker}.

Although the question on genericity no longer makes sense,
by adapting Theorem 4.3 of \cite{DHZ} to standard arguments of extensions,
one can prove the following statement.
\begin{Thm}\label{Thm:amecas}
Let $\Gamma$ be an amenable group.
Then there is a totally disconnected compact space $X$ with the following properties.
\begin{enumerate}[\upshape (1)]
\item $X$ has a clopen basis of cardinality at most $\sharp \Gamma$.
\item There is a minimal free action of $\Gamma$ on $X$ which admits
clopen tower decompositions with arbitrary invariance.
\end{enumerate}
\end{Thm}
\begin{proof}
Since the statement is trivial for finite groups, we assume that $\Gamma$ is infinite.
Let $\mathfrak{F}$ denote the set of finite subsets of $\Gamma$.
For each $F \in \mathfrak{F}$, denote $\Gamma_F$ the subgroup of $\Gamma$
generated by $F$.
For each $F \in \mathfrak{F}$,
we have a $\Gamma_{F}$-equivariant unital embedding
\[\iota_{F} \colon \ell^\infty (\Gamma_F) \rightarrow \ell^\infty(\Gamma)\]
obtained from the right coset decomposition of $\Gamma$ with respect to $\Gamma_F$. (We do not require the compatibility of maps $(\iota_F)_{F \in \mathfrak{F}}$ with respect to the inclusion.) 
For any $F\in \mathfrak{F}$ and $N \in \mathbb{N}$,
thanks to Theorem 4.3 of \cite{DHZ},
we have subsets $C_1, \ldots, C_n$ of $\Gamma_F$ and $(F, N^{-1})$-invariant finite subsets $S_1, \ldots, S_n$ of $\Gamma_F$ satisfying
$\Gamma_F= \bigsqcup_{i=1}^n \bigsqcup_{x_i \in S_i} x_i C_i$.
Define \[\mathcal{S}_{F, N}:=\{ \iota_{F}(\chi_{C_i}): i=1, \ldots, n\} \subset \ell^\infty (\Gamma).\]
Let $A$ be the \Cs-subalgebra of $\ell^\infty (\Gamma)$ generated by
$\mathcal{S} := \bigcup_{F, N} \Gamma \cdot \mathcal{S}_{F, N}$.
A standard calculation of cardinals shows that $\sharp \mathcal{S}\leq \sharp \Gamma$.
Let $Y$ denote the Gelfand dual of $A$.
Then the induced action $\alpha$ of $\Gamma$ on $Y$
forms the almost finite transformation groupoid.
One can find an extension $\beta$ of $Y$
satisfying the following conditions (cf.~\cite{RS} or the proof of Lemma 3.1 in \cite{Suz17}.)
\begin{enumerate}[(i)]
\item The $\beta$ is free.
\item 
The underlying space $Z$ of $\beta$ is totally disconnected and has a clopen basis of cardinality at most $\sharp\Gamma$.
\end{enumerate}
Let $W$ be a minimal closed $\Gamma$-invariant nonempty subset of $Z$.
Let $\gamma$ be the action obtained by restricting $\beta$ to $W$.
By Lemmas \ref{Lem:invsub} and \ref{Lem:quo}, the action
$\gamma$ posses the desired properties.
\end{proof}

Corollary \ref{Corint:universal} now follows from the Main Theorem and Theorem \ref{Thm:amecas}.

\appendix
\section{Pure infiniteness of crossed products for minimal extensions}\label{Sec:pi}
In this appendix, we adapt arguments in the proof of
Lemma \ref{Lem:zerodiv} to the study of pure infiniteness of the reduced groupoid \Cs -algebras.
As an application, we solve the question of pure infiniteness
of the reduced crossed product for minimal extensions.

Recall that a simple unital C$^\ast$-algebra $A$ is said to be purely infinite \cite{Cun}
if any nonzero positive element $a\in A$
admits an element $x\in A$ satisfying $x a x^\ast =1$.
The following theorem is the main result of this appendix.

\begin{Thm}\label{Thm:pi}
Let $G$ be a minimal groupoid
which admits a quotient homomorphism $\pi$ from $G$ onto an essentially principal groupoid $H$ satisfying the following conditions.
\begin{enumerate}[\upshape (1)]
\item $\pi$ is proper.
\item For each $s\in G^{(0)}$,
the restriction map $\pi|_{Gs} \colon Gs \rightarrow H\pi(s)$ is
bijective.
\item The reduced groupoid C$^\ast$-algebra ${\rm C}^\ast _{\mathrm r}(H)$ is purely infinite.
\end{enumerate}
Then the reduced groupoid C$^\ast$-algebra ${\rm C}^\ast _{\mathrm r}(G)$ is purely infinite.
\end{Thm}
\begin{proof}
By the first and second assumptions,
the pull-back map $\pi ^\ast \colon C_{\mathrm c}(H)\rightarrow C_{\mathrm c}(G)$ of $\pi$ defines a $\ast$-homomorphism.
As $\pi^\ast$ is compatible with the canonical conditional expectations,
it extends to an embedding
${\rm C}^\ast_{\mathrm r}(H)\rightarrow {\rm C}^\ast _{\mathrm r}(G)$.
Throughout the proof, we identify ${\rm C}^\ast _{\mathrm r}(H)$ with a C$^\ast$-subalgebra of ${\rm C}^\ast _{\mathrm r}(G)$
via this embedding.

Let $a \in {\rm C}^\ast _{\mathrm r}(G)$ be a positive element
satisfying $\|E(a)\|=1$.
Choose a positive element $a_0 \in C_{\mathrm c}(G)$ satisfying
$\|a- a_0\|< 1$ and $\|E(a_0) \|>1$.
Since $G$ is essentially principal (by assumptions (1) and (2)),
as in the proof of Lemma \ref{Lem:zerodiv},
one can find a nonnegative function $f\in C(G^{(0)})$
satisfying $\|f \| \leq 1$, $g:=fa_0 f\in C(G^{(0)})$, $\|g \|=1$, and
\[U:=\inte (\{s\in G^{(0)}: g(s)=1\})\neq \emptyset.\]

Fix a nonempty open subset $U_0$ of $U$
satisfying $\cl(U_0)\subset U$.
Take compact $H$-sets $\bar{V}_1, \ldots, \bar{V}_n, \bar{W}_1, \ldots, \bar{W}_n$
satisfying the following conditions:
$\bar{W_i}\subset \inte(\bar{V}_i)$ for all $i$
and $G^{(0)}= \bigcup _{i=1}^n r(W_iU_0)$, where $W_i:=\pi^{-1}(\bar{W_i})$ for each $i$.

Note that by assumptions (1) and (2), each $W_i$ is a compact $G$-set.
Since $H$ is essentially principal, one can find an element $s\in H^{(0)}$
satisfying
\[r^{-1}(\{s\})\cap s^{-1}(\{s\}) \cap \bar{V}_i \bar{V}_j^{-1} \subset \{s\}\]
for all $i, j$.
Put $S:= \pi^{-1}(\{s \})$.
Then the equalities $r(W_i)=\pi^{-1}(r(\bar{W_i}))$; $i\in \{1, \ldots, n\}$,
yield the relation $S\subset \bigcup_{i\in I} \tau_{W_i}(U_0)$,
where
\[I:= \left\{i\in \{1, \ldots, n\}: s \in r(\bar{W}_i)\right\}.\]
Take a maximal subset $J$ of $I$ satisfying the following property:
the elements $\tau_{\bar{W}_j}^{-1}(s); j\in J$ are pairwise distinct.
By the choice of $s$ and the third assumption, for any $i, j \in \{1, \ldots, n\}$,
\[ r^{-1}(S)\cap s^{-1}(S) \cap V_i V_j^{-1} \subset S,\]
where $V_i:=\pi^{-1}(\bar{V}_i)$ for each $i$.
From the above relations,
for any $i, j\in I$,
the equality $\tau_{\bar{W}_i}^{-1}(s)=\tau_{\bar{W}_j}^{-1}(s)$
shows the equality
\[\tau_{W_i}(U_0) \cap S =\tau_{W_j}(U_0) \cap S.\]
This yields the relation
$S\subset \bigcup_{j \in J} \tau_{W_j}(U_0)$.

Since the compact sets $\tau_{W_j}^{-1}(S); j\in J$, are pairwise disjoint, one can find nonnegative functions $\rho_j \in C(G^{(0)}); j \in J$
satisfying $\rho_j\rho_k=0$ for $j\neq k$,
$\rho_j \equiv 1$ on a neighborhood of $\tau_{W_j}^{-1}(S)\cap U_0$ for each $j\in J$,
and $\supp(\rho_j)\subset \cl(U)$ for all $j\in J$.
For each $j\in J$, take a function $\varphi_j \in C_{\mathrm c}(G)$
satisfying
$\varphi_j\equiv 1$ on a neighborhood of $W_j$ and $\supp(\varphi_j) \subset V_j$.
Set
\[v:=\sum_{j \in J} \varphi_j \rho_j.\]
Then, since $\rho_j=\rho_j g$ for all $j \in J$,
we have $vg=v$.
This implies
\[vgv^\ast=\sum_{j\in J} \varphi_j \rho_j^2 \varphi_j^\ast,\]
This equality shows that $vgv^\ast \in C(G^{(0)})$ and that
the inequality $(vgv^\ast)(t)\geq 1$ holds on a neighborhood $W$ of $S$.
Then, by the compactness of $G^{(0)}$,
one can choose a neighborhood $\bar{Z}$ of $s$
satisfying $Z:=\pi^{-1}(\bar{Z}) \subset W$.
Take a function $h\in C(G^{(0)})$ of norm at most one
satisfying $(hvgv^\ast h^\ast)(t)=1$ for all $t\in Z$.
Take functions $k_1, k_2 \in C(H^{(0)})$
of norm one satisfying $k_1 k_2=k_2$ and $\supp(k_1)\subset \bar{Z}$.
We then have
\[k_1hvgv^\ast h^\ast k_1^\ast=|k_1|^2 \in C(H^{(0)}).\]
Put $w:=k_1hv$. Then the above equality yields $w w^\ast=wgw^\ast=|k_1|^2$.

Since ${\rm C}^\ast_{\mathrm r} (H)$ is purely infinite, one can find
an element $y\in {\rm C}^\ast _{\mathrm r}(H)$
satisfying $y |k_2|^2 y^\ast =1$.
Set $x:=yk_2wf$.
Then, the equalities $k_1 k_2= k_2$ and $\|f\| =1$
yield
\[xx^\ast \leq yk_2 w w^\ast k_2^\ast y^\ast
=y|k_2|^2 y^\ast=1.\]
This shows $\| x\|\leq 1$.
Moreover, by the equality $k_1 k_2= k_2$,
we have
\[x a_0 x^\ast = yk_2wgw^\ast k_2^\ast y^\ast= y |k_2|^2 y^\ast =1.\]
This yields $\|xax^\ast -1\| <1$,
from which we conclude the invertibility of $xax^\ast$.
\end{proof}
As a special case of Theorem \ref{Thm:pi}, we obtain the following corollary
(cf.~a comment below Lemma \ref{Lem:quo}).
\begin{Cor}\label{Cor:pi}
Let $\Gamma$ be a group and $\alpha$ be a minimal and essentially free action of $\Gamma$
on a compact space whose reduced crossed product is purely infinite.
Then the reduced crossed product associated to any minimal extension of $\alpha$
is purely infinite.
\end{Cor}
Corollary \ref{Cor:pi} has the following three immediate applications.
As the first application, we give the following stronger version of Theorem 3.11 in \cite{RS}.
This is a new phenomenon in the study of amenable actions (\cite{Ana}).

\begin{Cor}\label{Cor:gRS}
Let $\Gamma$ be a countable non-amenable group.
Then $\Gamma$ admits a free minimal action $\alpha$ on the Cantor set
such that the reduced crossed products of all minimal extensions of $\alpha$
are purely infinite.
Moreover if $\Gamma$ is exact,
the action $\alpha$ can be amenable.
\end{Cor}
Since a statement similar to Theorem \ref{Thm:srK} holds true for purely infinite C$^\ast$-algebras (see Theorems 1.4 and 1.9 of \cite{Cun}), Corollary \ref{Cor:gRS} can also be seen as an analogue of
Corollary \ref{Corint:ame} for non-amenable groups.

The second application is the removal of a technical restriction of Theorem 3.12 in \cite{Suz15}.
Let us temporary introduce a class of compact spaces.
We first set
$\mathcal{C}'$ to be the class of compact metrizable spaces
which admits a minimal action of a (not necessary locally compact) path-connected group.
This class contains all closed topological manifolds and compact connected Hilbert manifolds (see \cite{GW}, \cite{Suz15}).
Moreover it is stable under taking the countable direct product.
We then define the class $\mathcal{C}$ to be the class of spaces
obtained as the direct product of the Cantor set and a space contained in $\mathcal{C}'$.
\begin{Cor}\label{Cor:Kir}
For any countable non-amenable exact group $\Gamma$ and any space $X$ contained in the class $\mathcal{C}$, there is an amenable minimal free action of $\Gamma$ on $X$
whose reduced crossed product is a Kirchberg algebra.
\end{Cor}
\begin{proof}
The statement immediately follows from Theorem 2.1 of \cite{Suz15}
and Corollary \ref{Cor:gRS}.
\end{proof}
The third application is on the universal minimal action.
\begin{Cor}\label{Cor:univpi}
For any non-amenable group,
the reduced crossed product of the universal minimal action is purely infinite.
\end{Cor}
\begin{proof}
By replacing sequences by appropriate nets in the arguments in \cite{RS}, one can prove the following claim.
Any non-amenable group
admits a minimal free action on a totally disconnected compact space
whose reduced crossed product is purely infinite.
The statement then follows from Corollary \ref{Cor:pi}.
\end{proof}
\section{Eigenvalue set for \`etale groupoids and constructions of distinguished minimal actions}\label{Sec:ES}
In this section, we introduce an isomorphism invariant of \`etale groupoids which we call the eigenvalue set. This invariant has already been used in Remark 3.12 of
\cite{Suz17}.
Since only a few isomorphism invariants are known for \`etale groupoids, this invariant itself would be of independent interest.
Here we use this invariant to establish the following theorem.
\begin{Thm}\label{Thm:examples}
Let $\Gamma$ be a countable non-torsion exact group.
Let $X$ be a space in the class $\mathcal{C}$ $($see above Corollary \ref{Cor:Kir} for the definition$)$.
Then there is a family of continuously many amenable minimal free actions of $\Gamma$
on $X$ whose transformation groupoids are pairwise non-isomorphic.
Moreover, one can additionally require the following properties to each member when
\begin{description}
 \item[$\Gamma$ is amenable] almost finiteness of the transformation groupoid,
 \item[$\Gamma$ is non-amenable]pure infiniteness of the reduced crossed product,
\item [$\Gamma$ is non-amenable and $X$ is the Cantor set]pure infiniteness of the transformation groupoid in Matui's sense $($Definition 4.9 in \cite{Mat15}$)$.
\end{description}
\end{Thm}

We recall that in this paper, groupoids are assumed to be \`etale, locally compact, with the compact unit space.

Let $\mathbb{T}$ be the group of complex numbers of modulus one.
Let $R$ denote the action of $\mathbb{T}$ on itself by the group multiplication.
The next definition is motivated by the analogous notion in the measured setting (see \cite{Gla} for instance).
\begin{Def}
Let $X$ be a compact space.
Let $h$ be a homeomorphism on $X$.
We say that an element $\lambda \in \mathbb{T}$
is an eigenvalue of $h$
if there is a continuous map $f \colon X \rightarrow \mathbb{T}$
satisfying $f\circ h= \lambda f$.
We denote $E(h)$ the set of eigenvalues of $h$.
\end{Def}
\begin{Def}
Let $G$ be a groupoid.
Define the subset $E(G)$ of $\mathbb{T}$
to be
\[E(G):= \bigcup_{\varphi \in [[G]]} E(\varphi).\]
Here $[[G]]$ denotes the topological full group of $G$
(see Section \ref{SubSec:Groupoids} for the definition).
We call $E(G)$ the eigenvalue set of $G$.
\end{Def}
Since an isomorphism of \`etale groupoids induces
the conjugacy of the actions of their topological full groups,
the eigenvalue set is an isomorphism invariant for \`etale groupoids.

Next we record a few basic properties of the eigenvalue set.
\begin{Lem}\label{Lem:ES}
Let $X$ be a second countable compact Hausdorff space.
Then for any homeomorphism $h$ on $X$,
the set $E(h)$ is countable.
\end{Lem}
\begin{proof}
For each $\lambda \in E(h)$,
take a unitary element $f_\lambda\in C(X)$
satisfying $f_\lambda\circ h=\lambda f_\lambda$.
We show that for any two distinct $\lambda, \mu \in E(h)$,
\[\|f_\lambda -f_\mu\|\geq 1.\]
To see this, note that for any $\nu, \eta \in \mathbb{T}$
with $\nu \neq 1$,
there is $n\in \mathbb{Z}$ satisfying $|\nu^n\eta - 1|\geq 1$.
Put $\zeta:=\lambda \bar{\mu}$.
Since
\[\|f_\lambda -f_\mu\|=\|(f_\lambda f_\mu^{\ast})\circ h^n - 1\|
=\|\zeta^nf_\lambda f_\mu^{\ast} - 1\|\]
for any $n \in \mathbb{Z}$, we obtain the desired inequality.

Since $C(X)$ is separable, the set $E(h)$ must be countable.
\end{proof}
\begin{Prop}\label{Prop:countability}
Let $G$ be a second countable groupoid.
Then its eigenvalue set $E(G)$ is countable.
\end{Prop}
\begin{proof}
Since $G$ is second countable,
there are only countably many compact open subsets in $G$.
This in particular shows that the topological full group $[[G]]$ is countable.
The claim now follows from Lemma \ref{Lem:ES}.
\end{proof}
\begin{Prop}\label{Prop:factor}
Let $G$ and $H$ be groupoids.
If there is a quotient homomorphism $\pi \colon G \rightarrow H$
satisfying
\begin{enumerate}[\upshape (1)]
\item $\pi$ is proper,
\item for each $s\in G^{(0)}$,
the restriction map $\pi|_{Gs} \colon Gs \rightarrow H\pi(s)$ is
bijective,
\item $\pi(G^{(0)})= H^{(0)}$,
\end{enumerate}
then $E(H) \subset E(G)$.
\end{Prop}
\begin{proof}
Observe that by the first two assumptions,
for any compact open $H$-set $V$,
the preimage $\pi^{-1}(H)$ is a compact open $G$-set.
This shows that for any $\varphi \in [[H]]$,
the group $[[G]]$ contains an extension $\tilde{\varphi}$ of $\varphi$.
This implies the desired inclusion.
\end{proof}

\begin{proof}[Proof of Theorem \ref{Thm:examples}]
We first prove the claim in the case that $X$ is the Cantor set.
Fix an element $t \in \Gamma$ of infinite order.
We first show that for any countable subset $S$ of $\mathbb{T}$,
there is an action $\alpha_S$ of $\Gamma$ on $X$ which satisfies the prescribed properties and the relation \[S \subset E(X \rtimes_{\alpha_S} \Gamma).\]
To show the claim, consider the product homeomorphism $R_S:=\prod_{s\in S} R_s$
on $\mathbb{T}^S$.
Let $Z$ be a minimal $R_S$-invariant closed subset of $\mathbb{T}^S$.
Let $T_S$ denote the restriction of $R_S$ on $Z$.
Obviously $T_S$ defines a minimal homeomorphism on $Z$.
It is clear that $E(T_S)$ contains $S$. 
Applying Lemma 3.1 of \cite{Suz17} to 
the action of $\Upsilon :=\{ t^n :n\in \mathbb{Z}\}$ on $Z$ given by $t\mapsto T_S$,
we obtain a minimal action $\alpha_S$ of $\Gamma$ on $X$
satisfying $S \subset E(\alpha_{S, t})$.
This yields the relation $S \subset E(X \rtimes _{\alpha_S} \Gamma)$.

From now on we divide the proof into two cases.
We first consider the case that $\Gamma$ is amenable.
Take a minimal free action $\beta$ of $\Gamma$ on the Cantor set $X$
whose transformation groupoid is almost finite \cite{Ker} (see also Theorem \ref{Thm:amecas}).
Take a minimal subsystem $\gamma_S$ of the diagonal action of $\alpha_S$ and $\beta$.
Then notice that the underlying space of $\gamma_S$
is homeomorphic to the Cantor set as it is totally disconnected, perfect, metrizable, and compact.
Clearly $\gamma_S$ factors onto both $\alpha_S$ and $\beta$.
Therefore, by Lemma \ref{Lem:quo}, the transformation groupoid of $\gamma_S$ is almost finite.
Also, by Proposition \ref{Prop:factor}, its eigenvalue set contains $S$.
Now standard calculations of cardinals show that the family 
$(\gamma_S)_S$ contains the desired family.

We next assume that the group $\Gamma$ is non-amenable.
It is not hard to conclude from the proof of Theorem 6.11 of \cite{RS} that
any minimal action of $\Gamma$ on a metrizable compact space
admits an amenable minimal extension on the Cantor set whose transformation groupoid is purely infinite.
The rest of the proof is similar to the amenable case.

The general case follows from
the Cantor set case, Proposition 2.1 of \cite{Suz15},
Lemma \ref{Lem:quo}, Corollary \ref{Cor:pi}, and Proposition \ref{Prop:factor}.
\end{proof}
\begin{Rem}
We do not know if one can further require families in Theorem \ref{Thm:examples} to give rise to pairwise non-isomorphic reduced crossed products.
This is the case for virtually free groups as shown in \cite{Suz13} and \cite{Suz15}.
\end{Rem}
\subsection*{Fundings}
This work was supported by the Japan Society for the Promotion of Science Research Fellowships for Young Scientists [PD 28-4705], the Japan Society for the Promotion of Science KAKENHI Grant-in-Aid for Young Scientists (Start-up) [No.~17H06737], and tenure track funds of Nagoya University.
\subsection*{Acknowledgement}
The author is grateful to the second reviewer of a journal
for his/her careful reading and valuable comments.

\end{document}